\documentclass[12pt,reqno]{amsart}

\usepackage{mathdots}

\usepackage{color}

\usepackage{pdfsync}

\usepackage{amssymb}

\usepackage{mathrsfs}

\DeclareMathAlphabet{\mathpzc}{OT1}{pzc}{m}{it}

\usepackage[all]{xy}

\usepackage{tikz}
\usetikzlibrary{arrows,decorations.pathmorphing,decorations.pathreplacing,positioning,shapes.geometric,shapes.misc,decorations.markings,decorations.fractals,calc,patterns}

\usepackage{float}

\usepackage[bottom]{footmisc}

\entrymodifiers={+!!<0pt,\fontdimen22\textfont2>}

\def\BZ{\mathbb{Z}}

\def\sC{\mathsf{C}}

\def\sR{\mathsf{R}}

\def\sT{\mathsf{T}}

\def\add{\operatorname{add}}

\def\adots{\mathinner{\mkern1mu\raise1.0pt\vbox{\kern7.0pt\hbox{.}}\mkern2mu\raise4.0pt\hbox{.}\mkern2mu\raise7.0pt\hbox{.}\mkern1mu}}

\def\ast{{\textstyle *}}

\def\coind{\operatorname{coind}}

\def\corad{\operatorname{corad}}

\def\dddots{\mathinner{\mkern1mu\raise10.0pt\vbox{\kern7.0pt\hbox{.}}\mkern2mu\raise5.3pt\hbox{.}\mkern2mu\raise1.0pt\hbox{.}\mkern1mu}}
\def\dddotssmall{\mathinner{\mkern1mu\raise7.0pt\vbox{\kern7.0pt\hbox{.}}\mkern-1mu\raise4pt\hbox{.}\mkern-1mu\raise1.0pt\hbox{.}\mkern1mu}}

\def\End{\operatorname{End}}

\def\fl{\mathsf{fl}}

\def\Gr{\operatorname{Gr}}

\def\Hom{\operatorname{Hom}}

\def\ind{\operatorname{ind}}
\def\indec{{\mathsf{ind}}}

\def\K{\operatorname{K}}

\def\Ker{\operatorname{Ker}}

\def\mod{\mathsf{mod}}
\def\Mod{\mathsf{Mod}}

\def\obj{\operatorname{obj}}

\def\rad{\operatorname{rad}}

\def\SL2{\operatorname{SL}_2}

\def\split{\operatorname{split}}

\numberwithin{equation}{section}

\newtheorem{Lemma}{Lemma}[section]
\newtheorem{Theorem}[Lemma]{Theorem}
\newtheorem{Proposition}[Lemma]{Proposition}

\theoremstyle{definition}
\newtheorem{Definition}[Lemma]{Definition}
\newtheorem{Setup}[Lemma]{Setup}

\newtheorem{Remark}[Lemma]{Remark}

\newtheorem{Example}[Lemma]{Example}

\newtheorem{bfhpg}[Lemma]{}               
\newtheorem*{bfhpg*}{}

\begin{document}

\setlength{\parindent}{0pt}
\setlength{\parskip}{7pt}

\title[Generalised friezes and modified Caldero-Chapoton,
II]{Generalised friezes and a modified Caldero-Chapoton map depending
  on a rigid object, II}

\author{Thorsten Holm}
\address{Institut f\"{u}r Algebra, Zahlentheorie und Diskrete
Mathematik, Fa\-kul\-t\"at f\"ur Mathematik und Physik, Leibniz
Universit\"{a}t Hannover, Welfengarten 1, 30167 Hannover, Germany}
\email{holm@math.uni-hannover.de}
\urladdr{http://www.iazd.uni-hannover.de/\~{ }tholm}

\author{Peter J\o rgensen}
\address{School of Mathematics and Statistics,
Newcastle University, Newcastle upon Tyne NE1 7RU, United Kingdom}
\email{peter.jorgensen@ncl.ac.uk}
\urladdr{http://www.staff.ncl.ac.uk/peter.jorgensen}




\begin{abstract} 

It is an important aspect of cluster theory that cluster categories
are ``categorifications'' of cluster algebras.  This is expressed
formally by the (original) Caldero-Chapoton map $X$ which sends certain
objects of cluster categories to elements of cluster algebras.

\medskip
\noindent
Let $\tau c \rightarrow b \rightarrow c$ be an Auslander-Reiten
triangle.  The map $X$ has the salient property that $X( \tau c )X( c ) -
X( b ) = 1$.  This is part of the definition of a so-called frieze,
see \cite{AD}.

\medskip
\noindent
The construction of $X$ depends on a cluster tilting object.  In a
previous paper \cite{HJ}, we introduced a modified Caldero-Chapoton
map $\rho$ depending on a rigid object; these are more general than
cluster tilting objects.  The map $\rho$ sends objects of sufficiently
nice triangulated categories to integers and has the key property that
$\rho( \tau c )\rho( c ) - \rho( b )$ is $0$ or $1$.  This is part of
the definition of what we call a generalised frieze.

\medskip
\noindent
Here we develop the theory further by constructing a modified
Caldero-Chapoton map, still depending on a rigid object, which sends
objects of sufficiently nice triangulated categories to elements of a
commutative ring $A$.  We derive conditions under which the map is a
generalised frieze, and show how the conditions can be satisfied if
$A$ is a Laurent polynomial ring over the integers.

\medskip
\noindent
The new map is a proper generalisation of the maps $X$ and $\rho$.

\end{abstract}

\maketitle

\setcounter{section}{-1}
\section{Introduction}
\label{sec:introduction}

The (original) Caldero-Chapoton map $X$ is an important object in
cluster theory.  The arguments of $X$ are certain objects of a cluster category, and the values are the corresponding elements of a cluster algebra.  The map $X$ expresses
that the cluster category is a categorification of the
cluster algebra, see \cite{CC}, \cite{CK}, \cite{CK2}, \cite{FK},
\cite{Palu}.  For example, Figure \ref{fig:AR_quiver} shows the
Auslander-Reiten (AR) quiver of $\sC( A_5 )$, the cluster category of
Dynkin type $A_5$, with a useful ``coordinate system''.  Figure
\ref{fig:frieze} shows the AR quiver again, with the values of $X$ on
the indecomposable objects of $\sC( A_5 )$.  The values are Laurent
polynomials over $\BZ$; indeed, the cluster algebra consists of such
Laurent polynomials.
\begin{figure}
\[
  \xymatrix @+1.8pc @!0 {
    *+[blue]{\{\, 5,7 \,\}} \ar[dr] && \{\, 6,8 \,\} \ar[dr] && *+[blue]{\{\, 1,7 \,\}} \ar[dr] && \{\, 2,8 \,\} \ar[dr] && \{\, 1,3 \,\} \ar@{.}[dd] \\
    & \{\, 5,8 \,\} \ar[dr] \ar[ur] && \{\, 1,6 \,\} \ar[dr] \ar[ur] && *+[red]{\{\, 2,7 \,\}} \ar[dr] \ar[ur] && \{\, 3,8 \,\} \ar[dr] \ar[ur] \\
    \{\, 4,8 \,\} \ar[dr] \ar[ur] \ar@{.}[uu] \ar@{.}[dd] && \{\, 1,5 \,\} \ar[dr] \ar[ur] && \{\, 2,6 \,\} \ar[dr] \ar[ur] && \{\, 3,7 \,\} \ar[dr] \ar[ur] && \{\, 4,8 \,\} \ar@{.}[uu] \ar@{.}[dd] \\
    & \{\, 1,4 \,\} \ar[dr] \ar[ur] && *+[red]{\{\, 2,5 \,\}} \ar[dr] \ar[ur] && \{\, 3,6 \,\} \ar[dr] \ar[ur] && \{\, 4,7 \,\} \ar[dr] \ar[ur] \\
    \{\, 1,3 \,\} \ar[ur] && *+[blue]{\{\, 2,4 \,\}} \ar[ur] && \{\, 3,5 \,\} \ar[ur] && \{\, 4,6 \,\} \ar[ur] && *+[blue]{\{\, 5,7 \,\}} \\
                        }
\]
\caption{The Auslander-Reiten quiver of the cluster category $\sC( A_5
  )$.  The dotted lines should be identified with opposite
  orientations.  The red vertices show the direct summands of a rigid
  object $R$, and the red and blue vertices show the direct summands
  of a cluster tilting object $T$.} 
\label{fig:AR_quiver}
\end{figure}
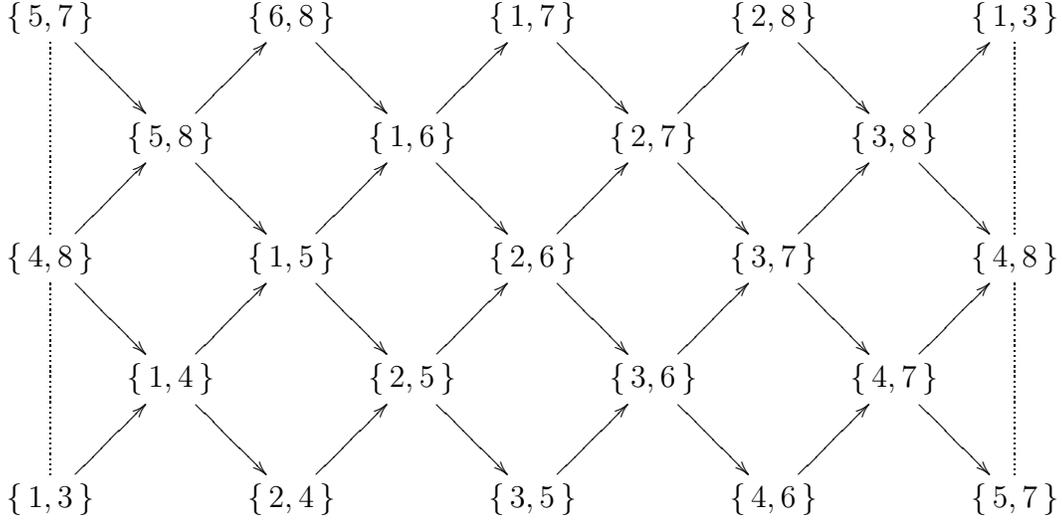
\begin{figure}
\[
  \xymatrix @+1.8pc @!0 {
    z \ar[dr]&& \frac{ ux + uy + yz + z }{ uyz } \ar[dr]&& u \ar[dr]&& \frac{ y + 1 }{ u } \ar[dr]&& \frac{ uvx + vz + xy + y }{ vxy } \\
    & \frac{ ux + yz + z }{ uy } \ar[dr] \ar[ur]&& \frac{ ux + uy + z }{ yz } \ar[dr] \ar[ur]&& y \ar[dr] \ar[ur]&& \frac{ uvx + vyz + vz + xy + xy^2 + y+y^2 }{ uvxy } \ar[dr] \ar[ur]\\
    \frac{ uvx + vyz + vz + y + y^2 }{ uxy } \ar[dr] \ar[ur]\ar@{.}[uu] \ar@{.}[dd] && \frac{ ux + z }{ y } \ar[dr] \ar[ur]&& \frac{ x + y }{ z } \ar[dr] \ar[ur]&& \frac{ vz + xy +  y }{ vx } \ar[dr] \ar[ur]&& \frac{ uvx + vyz + vz + y + y^2 }{ uxy } \ar@{.}[uu] \ar@{.}[dd] \\
    & \frac{ uvx + vz + y }{ xy } \ar[dr] \ar[ur]&& x \ar[dr] \ar[ur]&& \frac{ vz + x + x^2 + xy + y }{ vxz } \ar[dr] \ar[ur]&& \frac{ vz + y }{ x } \ar[dr] \ar[ur]\\
    \frac{ uvx + vz + xy + y }{ vxy }\ar[ur] && v \ar[ur]&& \frac{ x + 1 }{ v } \ar[ur]&& \frac{ vz + x + y }{ xz } \ar[ur]&& z \\
                        }
\]
\caption{The Auslander-Reiten quiver of $\sC( A_5 )$ with values of
  the original Caldero-Chapoton map $X$.  The map depends on the
  cluster tilting object $T$ shown by red and blue vertices in Figure
  \ref{fig:AR_quiver}.}
\label{fig:frieze}
\end{figure}
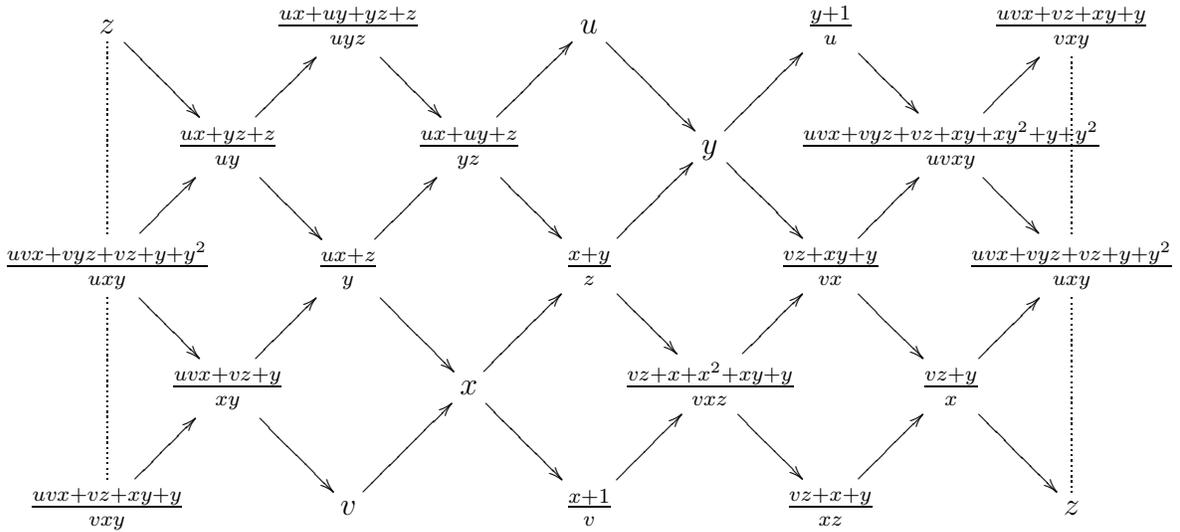

It is a salient property of $X$ that it is a {\em frieze} in
the sense of \cite{AD}, that is, if $\tau c \rightarrow b \rightarrow
c$ is an AR triangle then
\[
  X( \tau c )X( c ) - X( b ) = 1,
\]
see \cite[theorem]{DG} and \cite[prop.\ 3.10]{CC}.  In the case of
$\sC( A_5 )$, this means that for each ``diamond'' in the AR quiver,
of the form
\begin{equation}
\label{equ:diamond}
\vcenter{
  \xymatrix @-1.2pc {
    & b_1 \ar[dr] & \\ \tau c \ar[ur] \ar[dr] & & c \lefteqn{,} \\ & b_2 \ar[ur] &
                    }
        }
\end{equation}
we have
\[
  X( \tau c )X( c ) - X( b_1 )X( b_2 ) = 1.
\]

The definition of $X$ depends on a cluster tilting object $T$.  For
instance, the $X$ shown in Figure \ref{fig:frieze} depends on the $T$
which has the indecomposable summands shown by red and blue vertices in
Figure \ref{fig:AR_quiver}.

This paper is about a modified Caldero-Chapoton map $\rho$ which is
more general than $X$ in two respects: it depends on a rigid object
$R$ and has values in a general commutative ring $A$.  An object $R$
is rigid if $\Hom( R,\Sigma R ) = 0$.  This is much weaker than being
cluster tilting: recall that $T$ is cluster tilting if $\Hom( T ,
\Sigma t ) = 0 \Leftrightarrow t \in \add T \Leftrightarrow \Hom( t ,
\Sigma T ) = 0$.  Our first main result gives conditions under which
$\rho$ is a {\em generalised frieze}, in the sense that if $\tau c
\rightarrow b \rightarrow c$ is an AR triangle then
\[
  \rho( \tau c )\rho( c ) - \rho( b ) \in \{\, 0,1 \,\}.
\]
Our second main result is that the conditions can be satisfied if
$A$ is chosen to be a Laurent polynomial ring over the integers.

Generalised friezes with values in the integers were introduced by
combinatorial means in \cite{BHJ}, and it was shown in \cite{HJ} that
they can be recovered from a modified Caldero-Chapoton map.  The
theory of \cite{HJ} and the original Caldero-Chapoton map are both special cases of the theory developed here.

For example, consider $\sC( A_5 )$ again and let $R$ be the rigid
object which has the indecomposable summands shown by red vertices in
Figure \ref{fig:AR_quiver}.  Our results imply that we can choose $A =
\BZ[ u^{ \pm 1 } , v^{ \pm 1 } , z^{ \pm 1 } ]$, and Figure
\ref{fig:generalised_frieze} shows the AR quiver of $\sC( A_5 )$ with
the values of $\rho$ on the indecomposable objects.
\begin{figure}
\[
  \xymatrix @+1.8pc @!0 {
    z \ar[dr]\ar@{.}[dd] && \frac{ u+z }{ uz } \ar[dr] && u \ar[dr]&& \frac{ 1 }{ u } \ar[dr]&& \frac{ 1+uv+vz }{ v } \ar@{.}[dd] \\
    \begin{tikzpicture}[xscale=0.5,yscale=0.5,baseline=-0.5ex] \fill[gray!25] (-1,0) -- (0,1) -- (1,0) -- (0,-1) -- cycle; \end{tikzpicture}& \frac{ u+z }{ u } \ar[dr] \ar[ur]&& \frac{ u+z }{ z } \ar[dr] \ar[ur]&\begin{tikzpicture}[xscale=0.5,yscale=0.5,baseline=-0.5ex] \fill[gray!25] (-1,0) -- (0,1) -- (1,0) -- (0,-1) -- cycle; \end{tikzpicture}& 1 \ar[dr] \ar[ur]&\begin{tikzpicture}[xscale=0.5,yscale=0.5,baseline=-0.5ex] \fill[gray!25] (-1,0) -- (0,1) -- (1,0) -- (0,-1) -- cycle; \end{tikzpicture}& \frac{ 1+uv+vz }{ uv } \ar[dr] \ar[ur]\\
    \frac{ 1+uv+vz }{ u } \ar[dr] \ar[ur]\ar@{.}[dd] && u+z \ar[dr] \ar[ur]&& \frac{ 1 }{ z } \ar[dr] \ar[ur] && \frac{ 1+vz }{ v } \ar[dr] \ar[ur]&& \frac{ 1+uv+vz }{ u } \ar@{.}[dd] \\
    & 1+uv+vz \ar[dr] \ar[ur]&\begin{tikzpicture}[xscale=0.5,yscale=0.5,baseline=-0.5ex] \fill[gray!25] (-1,0) -- (0,1) -- (1,0) -- (0,-1) -- cycle; \end{tikzpicture}& 1 \ar[dr] \ar[ur]&\begin{tikzpicture}[xscale=0.5,yscale=0.5,baseline=-0.5ex] \fill[gray!25] (-1,0) -- (0,1) -- (1,0) -- (0,-1) -- cycle; \end{tikzpicture}& \frac{ 1+vz }{ vz } \ar[dr] \ar[ur]&& 1+vz \ar[dr] \ar[ur]&\begin{tikzpicture}[xscale=0.5,yscale=0.5,baseline=-0.5ex] \fill[gray!25] (-1,0) -- (0,1) -- (1,0) -- (0,-1) -- cycle; \end{tikzpicture} \\
    \frac{ 1+uv+vz }{ v } \ar[ur]&& v \ar[ur]&& \frac{ 1 }{ v } \ar[ur]&& \frac{ 1+vz }{ z } \ar[ur] && z \\
                        }
\]
\caption{The Auslander-Reiten quiver of $\sC( A_5 )$ with values of
  the modified Caldero-Chapoton map $\rho$.  The map depends on the
  rigid object $R$ shown by red vertices in Figure
  \ref{fig:AR_quiver}.  The values form a generalised frieze, the grey
  diamonds indicating where $\rho( \tau c )\rho( c ) - \rho( b_1
  )\rho( b_2 ) = 1$.}
\label{fig:generalised_frieze}
\end{figure}
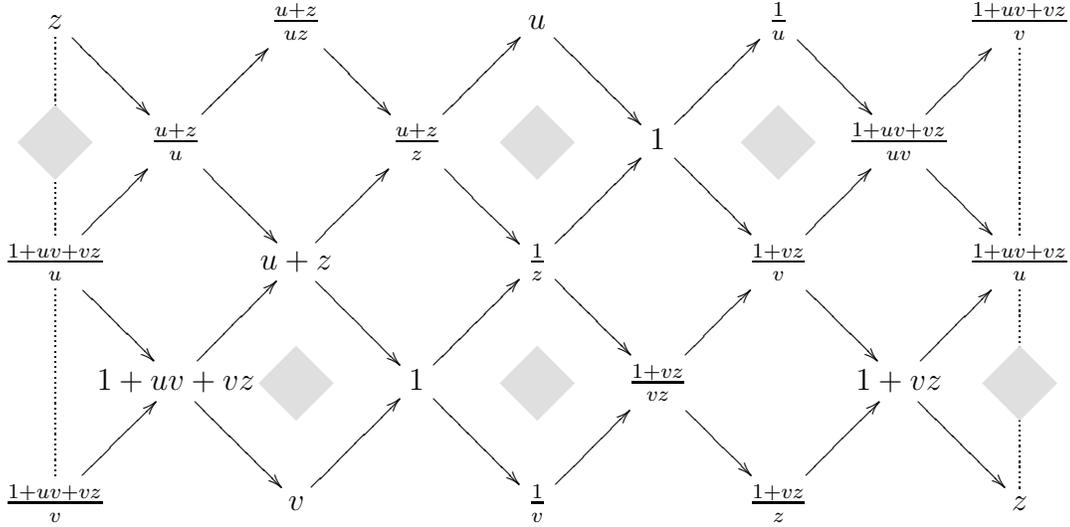
In this case, the generalised frieze property means that for
each ``diamond'' in the AR quiver, of the form \eqref{equ:diamond}, we
have
\[
  \rho( \tau c )\rho( c ) - \rho( b_1 )\rho( b_2 )
  \in \{\, 0,1 \,\}.
\]
The solid grey diamonds in Figure \ref{fig:generalised_frieze} indicate where the displayed expression is equal to $1$.

Let us explain how $\rho$ is defined.  Let $k$ be an algebraically
closed field, $\sC$ an essentially small $\Hom$-finite $k$-linear
triangulated category.  Assume that $\sC$ has split idempotents and
has a Serre functor.  Note that these are the only assumptions on
$\sC$ which is hence permitted to be a good deal more general than a
cluster category.  Let $\Sigma$ denote the suspension functor of $\sC$
and write $\sC( -,- )$ instead of $\Hom_{ \sC }( -,- )$.  Let $R$ be a
rigid object of $\sC$, assumed to be basic for reasons of simplicity,
with endomorphism algebra $E = \End_{ \sC }( R )$.  There is a functor
\[
  \begin{array}{rcl}
    \sC & \stackrel{G}{\longrightarrow} & \mod\,E, \\[2mm]
    c   & \longmapsto                   & \sC( R,\Sigma c ).
  \end{array}
\]
Let $A$ be a commutative ring and let
\[
  \alpha : \obj\,\sC \rightarrow A
\;\;,\;\;
  \beta : \K_0( \mod\,E ) \rightarrow A
\]
be two maps, where $\obj\,\sC$ is the set of objects of $\sC$ and
$\K_0$ denotes the Grothendieck group of an abelian category.

The modified Caldero-Chapoton map is the map $\rho : \obj\,\sC
\rightarrow A$ defined as follows.
\[
  \rho( c )
  = \alpha( c )
    \sum_e \chi \big( \Gr_e ( Gc ) \big) \beta( e )
\]
Here $c \in \sC$ is an object, the sum is over $e \in \K_0( \mod\,E
)$, and $\chi$ denotes the Euler characteristic defined by \'{e}tale
cohomology with proper support.  By $\Gr_e$ is denoted the
Grassmannian of submodules $M \subseteq Gc$ with $\K_0$-class $[ M ] =
e$. 

The original Caldero-Chapoton map is the special case where $R$ is a
cluster tilting object and $\alpha$ and $\beta$ are two particular
maps; see Remark \ref{rmk:original_CC}.  The modified Caldero-Chapoton
map of \cite{HJ} is the special case where $A = \BZ$ and $\alpha$ and
$\beta$ are identically equal to $1$.  In general, $\rho$ is only
likely to be interesting if the maps $\alpha$ and $\beta$ are chosen
carefully, and we formalise this by saying that $\alpha$ and $\beta$
are {\em frieze-like for the AR triangle} $\Delta = \tau c \rightarrow
b \rightarrow c$ if they satisfy the technical conditions in
Definition \ref{def:good}.  Observe that the conditions are trivially
satisfied if $\alpha$ and $\beta$ are identically equal to $1$.  Our
first main result is the following.

\bigskip
{\bf Theorem A. }
{\em 
If $\alpha$ and $\beta$ are frieze-like for each AR triangle in $\sC$,
then the modified Caldero-Chapoton map $\rho : \obj\,\sC \rightarrow
A$ is a generalised frieze in the sense of \cite[def.\ 3.4]{HJ}.  That
is,
\begin{enumerate}

  \item  $\rho( c_1 \oplus c_2 ) = \rho( c_1 )\rho( c_2 )$,

\medskip

  \item  if $\Delta  = \tau c \rightarrow b \rightarrow c$ is an AR
    triangle then $\rho( \tau c )\rho( c ) - \rho( b ) \in \{\, 0,1
    \,\}$.

\end{enumerate}
}
\bigskip

This follows from Proposition \ref{pro:exponential} and Theorem
\ref{thm:A} which even show
\[
  \rho( \tau c )\rho( c ) - \rho( b )
  = \left\{
      \begin{array}{cl}
        0 & \mbox{ if $G( \Delta )$ is a split short exact sequence, } \\[2mm]
        1 & \mbox{ if $G( \Delta )$ is not a split short exact sequence. }
      \end{array}
    \right.
\]
Note that $G( \Delta )$ is never split exact when $R$ is a
cluster tilting object, that is, in the case of the original
Caldero-Chapoton map.

Our second main result is that one can find frieze-like maps $\alpha$
and $\beta$, and hence generalised friezes, with values in Laurent
polynomials.

\bigskip
{\bf Theorem B. }
{\em
Assume that $\sC$ is $2$-Calabi-Yau and that the basic rigid object
$R$ has $r$ indecomposable summands and is a direct summand of a
cluster tilting object.

Then there are maps $\alpha$ and $\beta$ with values in $\BZ[ x_1^{
  \pm 1 }, \ldots, x_r^{ \pm 1 } ]$, using all the variables $x_1$,
$\ldots$, $x_r$, which are frieze-like for each AR triangle in $\sC$.

Hence there is a modified Caldero-Chapoton map $\rho : \obj \sC
\rightarrow \BZ[ x_1^{ \pm 1 }, \ldots, x_r^{ \pm 1 } ]$, using
all the variables $x_1$, $\ldots$, $x_r$, which is a generalised
frieze. 
}
\bigskip

This is established in Definition \ref{def:alpha_and_beta}, Theorem
\ref{thm:B}, and Remark \ref{rmk:app}.  We leave it vague for now what
it means to ``use all the variables $x_1, \ldots, x_r$'', but see
Remark \ref{rmk:app}.  In fact, it is sometimes possible to get values
in Laurent polynomials in more than $r$ variables.  For example, the
basic rigid object $R$ defined by the red vertices in Figure
\ref{fig:AR_quiver} has $r = 2$ indecomposable summands, but the
corresponding $\rho$ has values in $\BZ[ u^{ \pm 1 }, v^{ \pm 1 },
z^{ \pm 1 } ]$ as shown in Figure \ref{fig:generalised_frieze}.

It is natural to ask if the $\BZ$-algebra generated by the values
of $\rho$ is an interesting object.  In particular, one can ask how it
is related to cluster algebras and how it is affected by mutation of
rigid objects as defined in \cite[sec.\ 2]{MP}; see the questions in
Section \ref{sec:questions}.

The paper is organised as follows: Section \ref{sec:conditions}
defines what it means for $\alpha$ and $\beta$ to be frieze-like and
proves Theorem A.  Section \ref{sec:construction} proves Theorem B.
Section \ref{sec:example} shows how to obtain the example in Figure
\ref{fig:generalised_frieze}.  Section \ref{sec:questions} poses some
questions.

\section{The frieze-like condition on the maps $\alpha$ and $\beta$
implies that $\rho$ is a generalised frieze}
\label{sec:conditions}

This section proves Theorem A in the introduction.  It is a consequence of Proposition
\ref{pro:exponential} and Theorem \ref{thm:A}.

We start by setting up items to be used in the rest of the paper.
Instead of a basic rigid object $R$ we will work with a rigid
subcategory $\sR$.  This is more general since we can set $\sR =
\add\,R$ when $R$ is given, but not every $\sR$ has this form, see
\cite[sec.\ 6]{JP}.  The higher generality means that the definitions
of $G$ and $\beta$ are different from those in the introduction.

\begin{Setup}
\label{set:blanket}
Let $k$ be an algebraically closed field and let $\sC$ be an
essentially small $k$-linear $\Hom$-finite triangulated category with
split idempotents.  Hence $\sC$ is a Krull-Schmidt category.

Assume that $\sC$ has a Serre functor.  Hence it has AR triangles by
\cite[thm.\ A]{RVdB}, and the Serre functor is $\Sigma \circ \tau$
where $\Sigma$ is the suspension functor and $\tau$ is the AR
translation.

Let $\sR$ be a full subcategory of $\sC$ which is closed under direct
sums and summands, is functorially finite, and rigid, that is, $\sC(
\sR,\Sigma \sR ) = 0$.

The category of $k$-vector spaces is denoted $\Mod\,k$ and the
category of $k$-linear contravariant functors $\sR \rightarrow
\Mod\,k$ is denoted $\Mod\,\sR$.  It is a $k$-linear abelian category
and its full subcategory of objects of finite length is denoted
$\fl\,\sR$.

Let $A$ be a commutative ring and let
\[
  \alpha : \obj\,\sC \rightarrow A
\;\;,\;\;
  \beta : \K_0( \fl\,\sR ) \rightarrow A
\]
be maps which are ``exponential'' in the sense that
\begin{align*}
  \alpha( 0 ) = 1 \;\; , \;\;
  & \alpha( c \oplus d ) = \alpha( c )\alpha( d ), \\[2mm]
  \beta( 0 ) = 1 \;\; , \;\;
  & \beta( e + f ) = \beta( e )\beta( f ).
\end{align*}
\end{Setup}

\begin{bfhpg}
[The modified Caldero-Chapoton map]
\label{bfhpg:CC}
There is a functor
\[
  \begin{array}{rcl}
    \sC & \stackrel{G}{\longrightarrow} & \Mod\,\sR, \\[2mm]
    c   & \longmapsto                   & \sC( -,\Sigma c ) |_{ \sR }.
  \end{array}
\]
The modified Caldero-Chapoton map is defined by the following
formula.
\begin{equation}
\label{equ:CC}
  \rho( c )
  = \alpha( c )
    \sum_e \chi \big( \Gr_e ( Gc ) \big) \beta( e )
\end{equation}
The sum is over $e \in \K_0( \fl\,\sR )$, and $\Gr_e( Gc )$ is the
Grassmannian of submodules $M \subseteq Gc$ where $M$ has finite
length in $\Mod\,\sR$ and class $[ M ] = e$ in $\K_0( \fl\,\sR )$.
The notation is otherwise as explained in the introduction.


The formula may not make sense for each $c \in \sC$, but it does make
sense if $Gc$ has finite length in $\Mod\, \sR$ since then $\Gr_e( Gc
)$ is finite-dimensional and non-empty for only finitely many values
of $e$; see \cite[1.6 and 1.8]{JP}.  When the formula makes
sense, it defines an element $\rho( c ) \in A$.  Note that
\[
  \rho( 0 ) = 1.
\]

\end{bfhpg}

\begin{Proposition}
\label{pro:exponential}
Let $a,c \in \sC$ be objects such that $Ga,Gc$ have finite length in
$\Mod\, \sR$.  Then $G( a \oplus c )$ has finite length in $\Mod\,
\sR$ and
\[
  \rho( a \oplus c ) = \rho( a )\rho( c ).
\]
\end{Proposition}

\begin{proof}
The statement about the length of $G( a \oplus c )$ is clear.

It is immediate that
\begin{equation}
\label{equ:product}
  \rho( a )\rho( c )
  = \alpha( a \oplus c )
    \sum_g 
    \Big( 
      \sum_{ e+f=g } \chi \big( \Gr_e ( Ga ) \times \Gr_f ( Gc ) \big)
    \Big) \beta( g )
  = ( \ast ).
\end{equation}
In \cite[sec.\ 2]{HJ} we considered a pair of morphisms $a \rightarrow
b \rightarrow c$ in $\sC$ and introduced auxiliary spaces $X_{ e,f }$.
If we set $a \rightarrow b \rightarrow c$ equal to the canonical
morphisms $a \rightarrow a \oplus c \rightarrow c$, then \cite[lem.\
2.4(i+v)]{HJ} and \cite[rmk.\ 2.5]{HJ} mean that we can compute as
follows.
\[
  ( \ast )
  = \alpha( a \oplus c )
    \sum_g \Big( \sum_{ e+f = g } \chi( X_{ e,f } ) \Big) \beta( g )
  = \alpha( a \oplus c )
    \sum_g \chi \Big( \Gr_g \big( G( a \oplus c ) \big) \Big)\beta( g )
  = \rho( a \oplus c )
\]
\end{proof}

\begin{Definition}
[Frieze-like $\alpha$ and $\beta$]
\label{def:good}
Let 
\[
  \Delta = \tau c \rightarrow b \rightarrow c
\]
be an AR triangle in $\sC$ and assume that $Gc$ and $G( \tau c )$ have 
finite length in $\Mod\, \sR$.  We say that {\em $\alpha$ and $\beta$
are frieze-like for $\Delta$} if the following hold.
\begin{enumerate}

  \item  If $c \not\in \sR \cup \Sigma^{ -1 }\sR$ and $G( \Delta )$
    is a split short exact sequence, then
\[
  \alpha( b ) = \alpha( c \oplus \tau c ).
\]

\medskip

  \item  If $c \not\in \sR \cup \Sigma^{ -1 }\sR$ and $G( \Delta )$
    is a non-split short exact sequence, or if $c = \Sigma^{-1 }r
    \in \Sigma^{ -1 }\sR$, then
\[
  \alpha( b ) = \alpha( c \oplus \tau c ) 
\;\;,\;\;
  \alpha( c \oplus \tau c )\beta \big( [ Gc ] \big) = 1.
\]

\medskip

  \item  If $c = r \in \sR$, then
\[
  \alpha( c \oplus \tau c ) = 1
\;\;,\;\;
  \alpha( b ) = \beta \big( [S_r] \big) 
\]
where $S_r \in \Mod\,\sR$ is the simple object supported at $r$, see
\cite[prop.\ 2.2]{AusRepII}.

\end{enumerate}
\end{Definition}

\begin{Remark}
Note that if $c \not\in \sR \cup \Sigma^{ -1 }\sR$ then $G( \Delta )$
is a (split or non-split) short exact sequence, see \cite[lem.\
1.12(iii)]{HJ}. 
\end{Remark}

\begin{Theorem}
\label{thm:A}
Let 
\[
  \Delta = \tau c \rightarrow b \rightarrow c
\]
be an AR triangle in $\sC$.  Assume that $Gc$ and $G( \tau c )$ have 
finite length in $\Mod\, \sR$ and that $\alpha$ and $\beta$ are
frieze-like for $\Delta$.  Then $Gb$ has finite length in $\Mod\, \sR$
and
\[
  \rho( \tau c )\rho( c ) - \rho( b )
  =
  \left\{
    \begin{array}{cl}
      0 & \mbox{ if $G( \Delta )$ is a split short exact sequence, } \\[2mm]
      1 & \mbox{ if $G( \Delta )$ is not a split short exact sequence. }
    \end{array}
  \right.
\]
\end{Theorem}

\begin{proof}
It is clear from the definition of $G$ that it is a homological
functor, so $G( \Delta )$ is an exact sequence (albeit not necessarily
short exact).  The statement about the length of $Gb$ follows.

We split into cases.  First some preparation:
setting $a = \tau c$ in Equation \eqref{equ:product} gives
\begin{equation}
\label{equ:product2}
  \rho( \tau c )\rho( c )
  = \alpha( c \oplus \tau c )
    \sum_g 
    \Big( 
      \sum_{ e+f=g } \chi \big( \Gr_e ( G( \tau c ) ) \times \Gr_f ( Gc ) \big)
    \Big) \beta( g )
    = ( \ast ).
\end{equation}
Moreover, we use the morphisms $a \rightarrow b \rightarrow c$ and
auxiliary spaces $X_{ e,f }$ from \cite[sec.\ 2]{HJ} again, this time
setting $a \rightarrow b \rightarrow c$ equal to the AR triangle
$\tau c \rightarrow b \rightarrow c$.  We can then use the results of
\cite[sec.\ 2]{HJ}.

Case (i):  $c \not\in \sR \cup \Sigma^{ -1 }\sR$ and $G( \Delta )$ is
a split short exact sequence.  

We start from Equation \eqref{equ:product2} and
compute as follows:
\[
  ( \ast )
  \stackrel{\rm (a)}{=}
    \alpha( b )
    \sum_g 
    \Big( 
      \sum_{ e+f=g } \chi ( X_{ e,f } )
    \Big) \beta( g ) \\
  \stackrel{\rm (b)}{=}
    \alpha( b )
    \sum_g \chi \big( \Gr_g( Gb ) \big) \beta( g ) \\
  = \rho(b),
\]
where (a) is by Definition \ref{def:good}(i) and \cite[lem.\
2.4(i+v)]{HJ}, and (b) is by \cite[rmk.\ 2.5]{HJ}.

Case (ii): $c \not\in \sR \cup \Sigma^{ -1 }\sR$ and $G( \Delta )$ is a
non-split short exact sequence.

We start from Equation \eqref{equ:product2} and compute as follows:
\begin{align*}
  ( \ast )
  & =
    \alpha( c \oplus \tau c )
    \Big\{
      \chi \big( \Gr_0 ( G(\tau c) ) \times \Gr_{[Gc]} ( Gc ) \big)
      \beta \big( [Gc] \big)  \\
      & \;\;\;\;\;\;\;\;\;\;\;\;\;\;\;\;\;\;\;\;\;
        + \sum_g \Big(
                   \sum_{ \substack{ e+f=g \\[1mm] (e,f) \neq (0,[Gc]) }}
                                \chi \big(
                                        \Gr_e ( G(\tau c) )
                                        \times \Gr_f ( Gc )
                                      \big)
                 \Big) \beta( g ) 
    \Big\} \\
  & \stackrel{\rm (c)}{=}
    \alpha( c \oplus \tau c ) 
    \Big\{ 
      \beta \big( [Gc] \big) 
      + \sum_g \Big(
                 \sum_{ \substack{ e+f=g \\[1mm] (e,f) \neq (0,[Gc]) }}
                   \chi \big( X_{ e,f } \big)
               \Big) \beta( g )  
    \Big\} \\
  & \stackrel{\rm (d)}{=}
    1 + \alpha(b) \sum_g \Big(
                           \sum_{ \substack{ e+f=g \\[1mm] (e,f) \neq (0,[Gc]) }}
                           \chi \big( X_{ e,f } \big)
                         \Big) \beta( g ) \\
  & \stackrel{\rm (e)}{=}
    1 + \alpha(b) \sum_g \Big(
                           \sum_{ e+f=g }
                           \chi \big( X_{ e,f } \big)
                         \Big) \beta( g ) \\
  & \stackrel{\rm (f)}{=}
    1 + \alpha(b) \sum_g \chi \big( \Gr_g( Gb ) \big) \beta( g ) \\
  & = 1 + \rho( b ),
\end{align*}
where (c) follows from \cite[lem.\ 2.4(ii)+(iv)+(v)]{HJ}, (d) is by
Definition \ref{def:good}(ii), (e) is by \cite[lem.\ 2.4(iii)]{HJ},
and (f) is by \cite[rmk.\ 2.5]{HJ}.

Case (iii): $c = \Sigma^{ -1 }r \in \Sigma^{ -1 }\sR$.  Then
$G( \Delta )$ is not a split short exact sequence, but
\begin{equation}
\label{equ:Pr}
  G( \Delta )
  = G( \tau c \rightarrow b \rightarrow c )
  = 0 \rightarrow \rad\,P_r \rightarrow P_r
\end{equation}
by \cite[lem.\ 1.12(i)]{HJ}, where $P_r = \sC( -,r ) \,|_{ \sR }$ is
the indecomposable projective object of $\Mod\,\sR$ associated with
$r$, see \cite[1.5]{HJ}.  In particular, we have $G( \tau c ) = 0$ whence
$\rho( \tau c ) = \alpha( \tau c )$, and this gives the first equality
in the following computation.
\begin{align*}
  \rho( \tau c )\rho( c )
  & = \alpha( \tau c )\alpha( c )
      \sum_f \chi \big( \Gr_f( Gc ) \big) \beta( f ) \\
  & = \alpha( c \oplus \tau c )
      \sum_f \chi \big( \Gr_f( Gc ) \big) \beta( f ) \\
  & = \alpha( c \oplus \tau c )
      \Big\{
        \chi \big( \Gr_{ [Gc] }( Gc ) \big) \beta \big( [Gc] \big)
        + \sum_{ f \neq [Gc] } \chi \big( \Gr_f( Gc ) \big) \beta( f )
      \Big\} \\
  & \stackrel{\rm (g)}{=}
    \alpha( c \oplus \tau c )
      \Big\{
        \beta \big( [Gc] \big)
        + \sum_{ f \neq [Gc] } \chi \big( \Gr_f( Gc ) \big) \beta( f )
      \Big\} \\
  & \stackrel{\rm (h)}{=}
    \alpha( c \oplus \tau c )
      \Big\{
        \beta \big( [Gc] \big)
        + \sum_{ f } \chi \big( \Gr_f( Gb ) \big) \beta( f )
      \Big\} \\
  & \stackrel{\rm (j)}{=}
    1 + \alpha(b) \sum_{ f } \chi \big( \Gr_f( Gb ) \big) \beta( f ) \\
  & = 1 + \rho( b )
\end{align*}
To see (g), note that for $M' \subseteq Gc$ we have
\begin{equation}
\label{equ:1.2}
  [ M' ] = [ Gc ] \Leftrightarrow M' = Gc
\end{equation}
by \cite[eq.\ (1.2)]{HJ}.  Hence $\Gr_{ [Gc] }( Gc )$ has only a
single point whence $\chi \big( \Gr_{ [Gc] }( Gc ) \big) = 1$.
To see (h), note that Equation \eqref{equ:Pr} says that $Gc$ is an
indecomposable projective object with radical $Gb$.  So the proper
submodules of $Gc$ are precisely all the submodules of $Gb$, whence
Equation \eqref{equ:1.2} implies
\[
  \Gr_f( Gb )
  = \left\{
      \begin{array}{cl}
        \Gr_f( Gc ) & \mbox{ for $f \neq [ Gc ]$, } \\[2mm]
        \emptyset   & \mbox{ for $f = [ Gc ]$ }
      \end{array}
    \right.
\]
and (h) follows.  Finally, (j) holds by Definition
\ref{def:good}(ii). 

Case (iv): $c = r \in \sR$.  Then $G( \Delta )$ is not a split short
exact sequence, but we have
\[
  G( \Delta )
  = G( \tau c \rightarrow b \rightarrow c )
  = I_r \rightarrow \corad\,I_r \rightarrow 0
\]
by \cite[lem.\ 1.12(ii)]{HJ}, where $I_r = \sC( -,\Sigma\tau r ) \,|_{
  \sR }$ is the indecomposable injective object of $\Mod\,\sR$
associated with $r$, see \cite[1.10]{HJ}, and $\corad$ denotes the
quotient by the socle.  Now proceed dually to Case (iii), replacing
Definition \ref{def:good}(ii) by Definition \ref{def:good}(iii).
%
\end{proof}

\section{A construction of frieze-like maps $\alpha$ and $\beta$ with
values in Laurent polynomials}
\label{sec:construction}

This section proves Theorem B in the introduction.  It is a consequence of Definition
\ref{def:alpha_and_beta}, Theorem \ref{thm:B}, and Remark
\ref{rmk:app}.

\begin{Setup}
\label{set:2}
We continue to work under Setup \ref{set:blanket} and 
henceforth add the assumption that $\sC$ is a $2$-Calabi-Yau category with a
cluster tilting subcategory $\sT$ which belongs to a cluster
structure in the sense of \cite[sec.\ II.1]{BIRS}, and which satisfies
$\sR \subseteq \sT$.


Note that the AR translation of $\sC$ is
\[
  \tau = \Sigma
\]  
and that the Serre functor is $\Sigma^2$.
\end{Setup}

\begin{Remark}
When $\sR$ is a rigid subcategory of $\sC$, it is often possible to
find a cluster tilting subcategory $\sT$ with $\sR \subseteq \sT$.
Not always, however: if $\sC$ is a cluster tube, then such a $\sT$
cannot be found since cluster tubes have no cluster tilting
subcategories, see \cite[cor.\ 2.7]{BMV}.
\end{Remark}

\begin{bfhpg}
[Mutation and exchange triangles]
Let $\indec\, \sT$ denote the set of (isomorphism classes of)
indecomposable objects of $\sT$.  Each $t \in \indec\, \sT$ has a
mutation $t^{ \ast }$ which is the unique indecomposable object in
$\sC$ such that $\sT$ remains a cluster tilting subcategory if $t$ is
replaced by $t^{ \ast }$.  There are distinguished triangles
\begin{equation}
\label{equ:exchange_triangles1}
  t^{ \ast } \rightarrow a \rightarrow t
\;\;,\;\;
  t \rightarrow a' \rightarrow t^{ \ast }
\end{equation}
with $a, a' \in \add\big( ( \indec\, \sT ) \setminus t \big)$, known
as exchange triangles, see \cite[sec.\ II.1]{BIRS}.  
\end{bfhpg}

\begin{Definition} 
[The subgroup $N$]
\label{def:N}
The split Grothendieck group of an additive category is denoted by
$\K_0^{ \split }$.  It has a relation $[ a \oplus b ] = [ a ] + [ b ]$
for each pair of objects $a,b$, where $[ a ]$ is the $\K_0^{ \split
}$-class of $a$.

Define a subgroup of $\K_0^{ \split }( \sT )$ as follows.
\begin{equation}
\label{equ:N}
  N = \bigg\langle\, [a] - [a'] 
      \,\bigg|
        \begin{array}{l}
          \mbox{ $s^{ \ast } \rightarrow a \rightarrow s$ \;,\;
                 $s \rightarrow a' \rightarrow s^{ \ast }$
                 are exchange } \\[2mm]
          \mbox{ triangles with 
                 $s \in \indec\, \sT \setminus \indec\, \sR$ }
        \end{array}
      \,\bigg\rangle
\end{equation}
Let
\[
  Q : \K_0^{ \split }( \sT ) \rightarrow \K_0^{ \split }( \sT ) / N
\]
denote the canonical surjection.
\end{Definition}

\begin{bfhpg}
[Simple objects, $\K$-theory, and the homomorphism $\overline{ \theta }$]
\label{bfhpg:simples_and_K}
The inclusion functor $i : \sR \rightarrow \sT$ induces an exact functor
\[
  i^{ \ast } : \Mod\,\sT \longrightarrow \Mod\,\sR
\;\;,\;\;
  i^{ \ast }( M ) = M \circ i.
\]
Each indecomposable object $t \in \indec\, \sT$ gives rise to a simple
object $\overline{S}_t \in \Mod\,\sT$, and each $r \in \indec\, \sR$
gives rise to a simple object $S_r \in \Mod\,\sR$, see \cite[prop.\
2.3(b)]{AusRepII}.  It is not hard to show
\[
  i^{ \ast }\overline{S}_t 
  = \left\{
      \begin{array}{cl}
        S_t & \mbox{ if $t \in \indec\, \sR$, } \\[2mm]
        0   & \mbox{ if $t \in \indec\, \sT \setminus \indec\, \sR$. }
      \end{array}
    \right.
\]
Since $i^{ \ast }$ is exact and sends simple objects to simple objects
or $0$, is preserves finite length so restricts to an exact functor
\[
  i^{ \ast } : \fl\,\sT \rightarrow \fl\,\sR.
\]
Let
\[
  \kappa : \K_0( \fl\,\sT ) \rightarrow \K_0( \fl\,\sR )
\]
be the induced homomorphism.  The source is a free group on the
classes $[ \overline{S}_t ]$ for $t \in \indec\, \sT$ and the target
is a free group on the classes $[ S_r ]$ for $r \in \indec\, \sR$.
The homomorphism $\kappa$ is surjective and given by
\begin{equation}
\label{equ:kappa}
  \kappa\big( [ \overline{S}_t ] \big)
  = \left\{
      \begin{array}{cl}
        [ S_t ] & \mbox{ if $t \in \indec\, \sR$, } \\[2mm]
        0       & \mbox{ if $t \in \indec\, \sT \setminus \indec\, \sR$. }
      \end{array}
    \right.
\end{equation}

There is a functor
\[
  \begin{array}{rcl}
    \sC & \stackrel{\overline{G}}{\longrightarrow} & \Mod\,\sT, \\[2mm]
    c   & \longmapsto                              & \sC( -,\Sigma c ) |_{ \sT },
  \end{array}
\]
and $i^{ \ast } \overline{G} = G$ where $G$ is the functor from
Subsection \ref{bfhpg:CC}.

We define a homomorphism as follows,
\begin{equation}
\label{equ:overlinetheta}
  \overline{\theta} : \K_0( \fl\,\sT ) 
    \rightarrow \K_0^{\split}( \sT )
  \;\;,\;\; \overline{ \theta }
            \big( [ \overline{S}_t ] \big) = [ a ] - [ a' ],
\end{equation}
where $a, a'$ come from the exchange triangles
\eqref{equ:exchange_triangles1}, see \cite[1.5(ii)]{JP}. 
\end{bfhpg}

\begin{bfhpg}
[The homomorphism $\theta$]
It is clear from Equations \eqref{equ:N}, \eqref{equ:kappa}, and \eqref{equ:overlinetheta}
that there is a unique homomorphism $\theta$ which makes the following
square commutative.
\begin{equation}
\label{equ:square}
\vcenter{
  \xymatrix {
    \K_0( \fl\,\sT ) \ar@{->>}_{\kappa}[d] \ar^-{ \overline{ \theta }}[r] & \K_0^{ \split }( \sT ) \ar@{->>}^{Q}[d]\\
    \K_0( \fl\,\sR ) \ar_-{ \theta }[r] & \K_0^{ \split }( \sT ) / N \\
                     }
        }
\end{equation}
\end{bfhpg}

\begin{bfhpg}
[Index and coindex]
\label{bfhpg:ind}
For $c \in \sC$ there is a distinguished triangle $t_1 \rightarrow t_0
\rightarrow c$ with $t_0, t_1 \in \sT$ by \cite[sec.\ 1]{DK}, and the
index $\ind c = [t_0] - [t_1]$ is a well-defined element of $\K_0^{
\split }( \sT )$.  Similarly there is a distinguished triangle $c
\rightarrow \Sigma^2 t^0 \rightarrow \Sigma^2 t^1$ with $t^0, t^1 \in
\sT$, and the coindex $\coind c = [t^0] - [t^1]$ is a well-defined
element of $\K_0^{ \split }( \sT )$. 
%
\end{bfhpg}

\begin{Definition}
[The maps $\alpha$ and $\beta$]
\label{def:alpha_and_beta}
Recall that $A$ is a commutative ring.  Let
\begin{equation}
\label{equ:epsilon_exponential}
  \varepsilon : \K_0^{ \split }( \sT ) / N \rightarrow A
\end{equation}
be a map which is ``exponential'' in the sense that
\[
  \varepsilon( 0 ) = 1
\;\;,\;\;
  \varepsilon( e+f ) = \varepsilon( e )\varepsilon( f ).
\]
Define
\[
  \alpha : \obj\,\sC \rightarrow A
\;\;,\;\;
  \beta : \K_0( \fl\,\sR ) \rightarrow A
\]
by
\begin{equation}
\label{equ:alpha_beta}
  \alpha( c ) = \varepsilon Q( \ind c )
\;\;,\;\;
  \beta( e ) = \varepsilon \theta( e ). 
\end{equation}
It is easy to see that $\alpha$ and $\beta$ satisfy the conditions in
Setup \ref{set:blanket}, and Equation \eqref{equ:CC} now defines a
modified Caldero-Chapoton map $\rho$ with values in $A$. 

Remark \ref{rmk:app} has further comments to the choice of $\varepsilon$ which is crucial to the properties of $\rho$. 
\end{Definition}

\begin{Remark}
\label{rmk:original_CC}
The definition of $\alpha$ and $\beta$ is motivated by the original
Caldero-Chapoton map which is recovered as follows when $\sR = \sT$:
in this case, $N = 0$ and $Q$ is the identity while $\theta =
\overline{ \theta }$, so Equations \eqref{equ:alpha_beta} read
\[
  \alpha( c ) = \varepsilon ( \ind c )
\;\;,\;\;
  \beta( e ) = \varepsilon \overline{ \theta }( e ).
\]
The group $\K_0^{ \split }( \sT )$ is free on the classes $[ t ]$ for
$t \in \ind\,\sT$.  Let $A$ be the Laurent polynomial ring on
generators $x_t$ for $t \in \ind\,\sT$ and set $\varepsilon \big( [ t
] \big) = x_t$ for $t \in \ind\,\sT$.  Then the map $\rho$ from
Equation \eqref{equ:CC} is the original Caldero-Chapoton map, see
\cite[1.8]{JP}. 
\end{Remark}

\begin{Lemma}
\label{lem:coind}
The map $\overline{\theta}$ from Equation \eqref{equ:overlinetheta}
satisfies the following. 
\begin{enumerate}

  \item  If $\overline{G}c$ has finite length in $\Mod\, \sT$ then
\[
  \overline{\theta} \big( [ \overline{G}c ] \big)
  = - ( \ind c + \ind \Sigma c).
\]

\smallskip

  \item Let $\Sigma c \stackrel{ \varphi }{ \longrightarrow } b
  \longrightarrow c$ be an AR triangle in $\sC$.  If $\overline{ G }(
  \Sigma c )$ and $\overline{ G }c$ have finite length in $\Mod\, \sT$
  then
\[
  \ind b
  = \left\{
      \begin{array}{cl}
        - \overline{ \theta } \big( [ \overline{ G }c ] \big)
          & \mbox{ if $c \not\in \sT$, } \\[2mm]
        \overline{ \theta } \big( [ \overline{S}_t ] \big)
          & \mbox{ if $c = t \in \sT$. }
      \end{array}
    \right.
\]
\end{enumerate}
\end{Lemma}

\begin{proof}
First note that \cite[lem.\ 2.1(2) and prop.\ 2.2]{Palu} apply to the present setup by \cite[1.3]{JP}.

(i)  Combine \cite[1.5]{JP} with \cite[lem.\ 2.1(2)]{Palu}.

(ii)  Observe that
\begin{equation}
\label{equ:Sigmab1}
  \ind b
  = \overline{ \theta } \big( [ \Ker \overline{ G }\varphi ] 
                              - [ \overline{ G }c ] \big).
\end{equation}
This can be seen by combining part (i) of the lemma with 
\cite[prop.\ 2.2]{Palu}.


The case $c \not\in \sT$: then $\sC( t,b ) \rightarrow \sC( t,c
)$ is surjective because $\Sigma c \stackrel{ \varphi }{
\longrightarrow } b \longrightarrow c$ is an AR triangle.  The long
exact sequence 
$\sC( t,b )
  \longrightarrow \sC( t,c )
  \longrightarrow \sC( t,\Sigma^2 c )
  \stackrel{ ( \Sigma \varphi )_* }{ \longrightarrow }
    \sC( t,\Sigma b )$
shows that $( \Sigma \varphi )_{ \ast } : \sC( t,\Sigma^2 c )
\rightarrow \sC( t,\Sigma b )$ is injective.  This implies that
$\overline{G}\varphi$ is injective whence Equation \eqref{equ:Sigmab1}
gives $\ind b = - \overline{ \theta } \big( [ \overline{ G }c ]
\big)$ as desired.

The case $c = t \in \sT$: then 
\[
  \overline{ G }( \Sigma c
  \stackrel{ \varphi }{ \longrightarrow } b
  \longrightarrow c )
  = \overline{ I }_t
  \stackrel{ \overline{ G }\varphi }{ \longrightarrow }
  \corad \overline{ I }_t
  \rightarrow
  0
\]
by \cite[lem.\ 1.12(ii)]{HJ}.  Here $\overline{ I }_t = \sC(
-,\Sigma^2 t )|_{ \sT }$ is the indecomposable injective object of
$\Mod\,\sT$ associated with $t$, and $\corad$ denotes the quotient by
the socle.  Hence $\Ker \overline{ G }\varphi = \overline{ S }_t$ and
$\overline{ G }c = 0$, whence Equation \eqref{equ:Sigmab1} reads $\ind
b = \overline{ \theta } \big( [ \overline{ S }_t ] \big)$ as desired.
\end{proof}

\begin{Theorem}
\label{thm:B}
Let 
\[
  \Delta 
   = \Sigma c \rightarrow b \rightarrow c
\]
be an AR triangle in $\sC$ such that $\overline{ G }c$ and
$\overline{ G }( \Sigma c )$ have finite length in $\Mod\, \sT$.

Then $Gc$ and $G( \Sigma c )$ have finite length in $\Mod\, \sR$, and
the maps $\alpha$ and $\beta$ from Definition \ref{def:alpha_and_beta}
are frieze-like for $\Delta$.
\end{Theorem}

\begin{proof}
The statement on lengths holds because $G = i^{ \ast }\overline{ G }$
and $i^{ \ast }$ preserves finite length, see Subsection
\ref{bfhpg:simples_and_K}.  We must now check the conditions of
Definition \ref{def:good}.

First, we show that
\begin{equation}
\label{equ:alpha2}
  \alpha( c \oplus \Sigma c )\beta \big( [ Gc ] \big) = 1,
\end{equation}
in particular establishing the second equation in Definition
\ref{def:good}(ii).  Equation \eqref{equ:alpha_beta} gives
\begin{equation}
\label{equ:alpha10}
  \alpha( c \oplus \Sigma c )
  = \varepsilon Q \big( \ind ( c \oplus \Sigma c ) \big)
  = \varepsilon Q( \ind c + \ind \Sigma c ).
\end{equation}
On the other hand, combining Equation \eqref{equ:alpha_beta}, the fact
that $[ Gc ] = [ i^{ \ast }\overline{ G }c ] = \kappa[ \overline{ G }c
]$, the commutative square \eqref{equ:square}, and Lemma
\ref{lem:coind}(i) gives
\[
  \beta \big( [ Gc ] \big)
  = \varepsilon \theta \big( [ Gc ] \big)
  = \varepsilon \theta \kappa \big( [ \overline{ G }c ] \big)
  = \varepsilon Q \overline{ \theta } \big( [ \overline{ G }c ] \big)
  = \varepsilon Q \big( - ( \ind c + \ind \Sigma c ) \big).
\]
Multiplying the last two equations proves Equation \eqref{equ:alpha2}. 

Secondly, if $c = t \in \sT$ then it is direct from the
definition of index and coindex that
\[
  \ind c = [ t ]
\;\;,\;\;
  \ind \Sigma c = - [ t ].
\]
Inserting into Equation \eqref{equ:alpha10} gives
\begin{equation}
\label{equ:alpha3}
  c \in \sT
  \; \Rightarrow \;
  \alpha( c \oplus \Sigma c ) = 1,
\end{equation}
in particular establishing the first equation in Definition
\ref{def:good}(iii).

Thirdly, suppose $c \not\in \sR$.  We will show
\begin{equation}
\label{equ:alpha}
  \alpha( b ) = \alpha( c \oplus \Sigma c ),
\end{equation}
establishing Definition \ref{def:good}(i) as well as the first
equation in Definition \ref{def:good}(ii).

The case $c = t \in \sT$: Note that $c \not\in \sR$ implies
$\overline{ \theta } \big( [ \overline{ S }_t ] \big) \in N$ by
Equations \eqref{equ:overlinetheta} and \eqref{equ:N}.  Using
Equation \eqref{equ:alpha_beta} and Lemma \ref{lem:coind}(ii)
therefore gives 
\[
  \alpha( b )
  = \varepsilon Q ( \ind b )
  = \varepsilon Q \overline{ \theta }\big( [ \overline{ S }_t ] \big)
  = \varepsilon( 0 )
  = 1.
\]
Combining with Equation \eqref{equ:alpha3} shows Equation
\eqref{equ:alpha}. 

The case $c \not\in \sT$: combining the two parts of Lemma
\ref{lem:coind} shows $\ind b = \ind c + \ind \Sigma c$.  Applying
$\varepsilon Q$ shows $\alpha( b ) = \alpha( c )\alpha( \Sigma c )$
which is equivalent to Equation \eqref{equ:alpha}.

Finally, suppose $c = r \in \sR$.  We show that
\[
  \alpha( b ) = \beta \big( [ S_r ] \big),
\]
establishing the second equation in Definition \ref{def:good}(iii).
Lemma \ref{lem:coind}(ii) says 
\[
  \ind b = \overline{ \theta } \big( [ \overline{ S }_r ] \big).
\]
Applying $\varepsilon Q$ gives the first of the following equalities. 
\[
  \alpha( b )
  = \varepsilon Q \overline{ \theta } \big( [ \overline{ S }_r ] \big)
  = \varepsilon \theta \kappa \big( [ \overline{ S }_r ] \big)
  = \varepsilon \theta \big( [ S_r ] \big)
  = \beta \big( [ S_r ] \big)
\]
The other equalities are by the commutative diagram \eqref{equ:square}
and Equations \eqref{equ:kappa} and \eqref{equ:alpha_beta}.
\end{proof}

\begin{Remark}
\label{rmk:app}
The maps $\alpha$ and $\beta$ from Definition
\ref{def:alpha_and_beta}, and hence the modified Caldero-Chapoton map
$\rho$ from Equation \eqref{equ:CC}, depend on the map $\varepsilon :
\K_0^{ \split }( \sT ) / N \rightarrow A$.  The possible choices of
$\varepsilon$ are determined by the structure of $\K_0^{ \split }( \sT
) / N$ which we do not know in general.

However, let us suppose that $\indec\, \sR$ and $\indec\, \sT$ are
finite, with $r$, respectively $r+s$, objects.  This is the situation from
Theorem B in the introduction if we set $R$, respectively $T$, equal
to the direct sum of the indecomposable objects in $\ind\,\sR$,
respectively $\ind\,\sT$.
Then we can set $A = \BZ[ x_1^{ \pm 1 }, \ldots, x_r^{ \pm 1
} ]$ and use all the variables $x_1, \ldots, x_r$, thereby proving
Theorem B.

Namely, $\K_0^{ \split }( \sT )$ is a free abelian group
on $r+s$ generators, one per object in $\indec\, \sT$, and the
subgroup $N$ has $s$ generators, one per object in $\indec\, \sT
\setminus \indec\, \sR$.  So $\K_0^{ \split }( \sT ) / N$ has a
quotient group $F$ which is free abelian of rank $( r+s ) - s = r$.
The desired map
$\varepsilon : \K_0^{ \split }( \sT ) / N \rightarrow \BZ[ x_1^{ \pm 1
}, \ldots, x_r^{ \pm 1 } ]$ can be obtained by sending each generator
of $F$ to a generator of $\BZ[ x_1^{ \pm 1 }, \ldots, x_r^{ \pm 1 }
]$.

Note that $\K_0^{ \split }( \sT )$ may have a quotient group which is
free abelian of rank $n > r$, see the example in Section
\ref{sec:example}.  In this case, the above method means that
we can even set $A = \BZ[ x_1^{ \pm 1 }, \ldots, x_n^{ \pm 1 } ]$ and
use all the variables $x_1, \ldots, x_n$.
\end{Remark}

\section{Example: a modified Caldero-Chapoton map on the cluster
category of Dynkin type $A_5$}
\label{sec:example}

This section shows how to obtain the example in Figure
\ref{fig:generalised_frieze} in the introduction.

\begin{Setup}
Let the category $\sC$ of Setups \ref{set:blanket} and \ref{set:2} be
$\sC( A_5 )$, the cluster category of Dynkin type $A_5$.  There is a
bijection between $\indec\, \sC$ and the diagonals of a $8$-gon, see
\cite[secs.\ 2 and 5]{CCS}.  We let $\indec\, \sR$, respectively
$\indec\, \sT$, be given by the red diagonals, respectively all the
red and blue diagonals, in Figure \ref{fig:8gon}.  These data satisfy
our assumptions, see \cite[sec.\ 1]{BMRRT}.
\begin{figure}
\[
  \begin{tikzpicture}[auto]
    \node[name=s, shape=regular polygon, regular polygon sides=8,
    minimum size=5cm, draw] at (0,0) {}; 
    \node[name=t, shape=regular polygon, regular polygon sides=8, minimum size=5.8cm] at (0,0) {}; 
    \draw[shift=(t.corner 1)] node {$1$};
    \draw[shift=(t.corner 2)] node {$2$};
    \draw[shift=(t.corner 3)] node {$3$};
    \draw[shift=(t.corner 4)] node {$4$};
    \draw[shift=(t.corner 5)] node {$5$};
    \draw[shift=(t.corner 6)] node {$6$};
    \draw[shift=(t.corner 7)] node {$7$};
    \draw[shift=(t.corner 8)] node {$8$};
    \draw[very thick, blue] (s.corner 1) to (s.corner 7);
    \draw[very thick, blue] (s.corner 2) to (s.corner 4);
    \draw[very thick, red] (s.corner 2) to (s.corner 5);
    \draw[very thick, red] (s.corner 2) to (s.corner 7);
    \draw[very thick, blue] (s.corner 5) to (s.corner 7);
  \end{tikzpicture} 
\]
\caption{The diagonals of the $8$-gon correspond to the indecomposable
objects of $\sC( A_5 )$.  The red diagonals define $\indec\, \sR$ and
all the red and blue diagonals define $\indec\, \sT$.}   
\label{fig:8gon}
\end{figure}
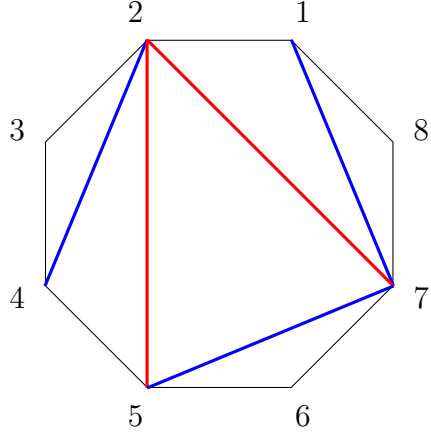
\end{Setup}

\begin{bfhpg}
[Some properties of $\sC$]
We denote diagonals and their corresponding indecomposable objects by
pairs of vertices, so $\{\, 2,7 \,\}$ is both a red diagonal in Figure
\ref{fig:8gon} and an object of $\indec\, \sR$.  The AR quiver of
$\sC$ and the objects of $\ind\,\sR$, respectively $\ind\,\sT$, are
shown in Figure \ref{fig:AR_quiver} in the introduction. 

At the level of objects, the suspension functor $\Sigma$ is given by
$\Sigma \{\, i,j \,\} = \{\, i-1,j-1 \,\}$.  Note that vertex numbers
are taken modulo $8$.

If $x, y \in \indec\, \sC$ then
\begin{equation}
\label{equ:C7_Homs}
  \sC( x,\Sigma y )
  =
  \left\{
    \begin{array}{cl}
      k & \mbox{ if the diagonals corresponding to $x$ and $y$ cross, } \\[2mm]
      0 & \mbox{ if not. }
    \end{array}
  \right.
\end{equation}
If $i,k,j,\ell$ are four vertices in anticlockwise
order on the polygon, then $\{\, i,j \,\}$ and $\{\, k,\ell \,\}$ are
crossing diagonals, and there are the following non-split distinguished
triangles,
\begin{equation}
\label{equ:exchange_triangles}
  \{\, i,j \,\}
  \rightarrow \{\, i,\ell \,\} \oplus \{\, j,k \,\}
  \rightarrow \{\, k,\ell \,\}
  \;\; , \;\;
  \{\, k,\ell \,\}
  \rightarrow \{\, i,k \,\} \oplus \{\, j,\ell \,\}
  \rightarrow \{\, i,j \,\},
\end{equation}
where a pair of neighbouring vertices must be interpreted as $0$.
\end{bfhpg}

\begin{bfhpg}
[$\K$-theory]
The category $\sT$ has the following indecomposable objects.
\[
  \{\, 1,7 \,\}
  \;\;,\;\; \{\, 2,4 \,\}
  \;\;,\;\; \{\, 2,5 \,\}
  \;\;,\;\; \{\, 2,7 \,\}
  \;\;,\;\; \{\, 5,7 \,\}
\]
Their $\K_0^{ \split }$-classes are free generators of $\K_0^{ \split
}( \sT )$.  To save parentheses, the classes are denoted $[ 1,7 ]$
etc.  The objects in $\indec\, \sT \setminus \indec\, \sR$ are
\[  
    \{\, 1,7 \,\} \;\;,\;\; \{\, 2,4 \,\} \;\;,\;\; \{\, 5,7 \,\},
\]
and Equation \eqref{equ:exchange_triangles} means that they sit in the
following exchange triangles.
\[
  \xymatrix @R=1ex {
    \{\, 2,8 \,\} \ar[r] & 0 \ar[r] & \{\, 1,7 \,\}
    & \{\, 1,7 \,\} \ar[r] & \{\, 2,7 \,\} \ar[r] & \{\, 2,8 \,\} \\
    \{\, 3,5 \,\} \ar[r] & 0 \ar[r] & \{\, 2,4 \,\}
    & \{\, 2,4 \,\} \ar[r] & \{\, 2,5 \,\} \ar[r] & \{\, 3,5 \,\} \\
    \{\, 2,6 \,\} \ar[r] & \{\, 2,7 \,\} \ar[r] & \{\, 5,7 \,\}
    & \{\, 5,7 \,\} \ar[r] & \{\, 2,5 \,\} \ar[r] & \{\, 2,6 \,\}
            }
\]
Accordingly, the subgroup $N$ of Definition \ref{def:N} is
\[
  N = \big\langle\,
        - [ 2,7 ] \;,\; - [ 2,5 ] \;,\; [ 2,7 ] - [ 2,5 ]
      \,\big\rangle \\[2mm]
    = \big\langle\,
        [ 2,5 ] \;,\; [ 2,7 ]
      \,\big\rangle,
\]
and $\K_0^{ \split }( \sT ) / N$ is the free abelian group generated
by
\[
  [ 1,7 ] + N
  \;\;,\;\; [ 2,4 ] + N
  \;\;,\;\; [ 5,7 ] + N.
\]

The category $\fl\, \sR$ has the simple objects
\[
  S_{\{\, 2,5 \,\}} \;\;,\;\; S_{\{\, 2,7 \,\}}
\]
whose $\K_0$-classes are free generators of $\K_0( \fl\, \sR )$, and
$\fl\, \sT$ has the simple objects
\[
  \overline{ S }_{\{\, 1,7 \,\}}
  \;\;,\;\; \overline{ S }_{\{\, 2,4 \,\}}
  \;\;,\;\; \overline{ S }_{\{\, 2,5 \,\}}
  \;\;,\;\; \overline{ S }_{\{\, 2,7 \,\}}
  \;\;,\;\; \overline{ S }_{\{\, 5,7 \,\}}
\]
whose $\K_0$-classes are free generators of $\K_0( \fl\, \sT )$.  To
save parentheses, the simple objects will be denoted $S_{ 2,5 }$,
respectively $\overline{ S }_{ 1,7 }$, etc.
\end{bfhpg}

\begin{Definition}
Let the map
\[
  \varepsilon : \K_0^{ \split }( \sT ) / N
                \rightarrow
                \BZ[ u^{ \pm 1 } , v^{ \pm 1 } , z^{ \pm 1 } ]
\]
be given by
\begin{equation}
\label{equ:epsilon}
  \varepsilon \big( [ 1,7 ] + N \big) = u    
\;\;,\;\;
  \varepsilon \big( [ 2,4 ] + N \big) = v    
\;\;,\;\;
  \varepsilon \big( [ 5,7 ] + N \big) = z.   
\end{equation}
Equation \eqref{equ:alpha_beta} defines the maps $\alpha$ and $\beta$, 
and the modified Caldero-Chapton map $\rho$ is defined by Equation
\eqref{equ:CC}.
\end{Definition}

\begin{Example}
Let us compute $\rho \big( \{\, 4,6 \,\} \big)$.

Equation \eqref{equ:alpha_beta} gives
\[
  \alpha \big( \{\, 4,6 \,\} \big)
  = \varepsilon Q( \ind \{\, 4,6 \,\} )
  = ( \ast ).
\]
Now $\{\, 4,6 \,\} = \Sigma \{\, 5,7 \,\}$ so $\ind \{\, 4,6 \,\} =
\ind \Sigma \{\, 5,7 \,\} = - [ 5,7 ]$, where the last equality is
direct from the definition of index because $\{\, 5,7 \,\} \in \sT$,
see Subsection \ref{bfhpg:ind}.  Hence Equation \eqref{equ:epsilon}
gives
\[
  ( \ast )
  = \varepsilon Q \big( -[ 5,7 ] \big)
  = \varepsilon \big( -[ 5,7 ] + N \big)
  = z^{ -1 }.
\]


We have $G \big( \{\, 4,6 \,\} \big) = \sC( -,\Sigma \{\, 4,6 \,\}
)|_{ \sR }$.  Moreover, $\sR = \add \big\{\, \{\, 2,5 \,\} , \{\, 2,7
\,\} \,\big\}$, and it is direct from Equation
\eqref{equ:C7_Homs} that $G \big( \{\, 4,6 \,\} \big)$ is supported
only at $\{\, 2,5 \,\}$ where it has the value $k$.  That is,
\[
  G \big( \{\, 4,6 \,\} \big) = S_{ 2,5 }.
\]
It follows that the only non-empty Grassmannians appearing in Equation
\eqref{equ:CC} when computing $\rho \big( \{\, 4,6 \,\} \big)$ are
$\Gr_0 \big( G\{\, 4,6 \,\} \big)$ and $\Gr_{ [ S_{ 2,5 } ] } \big(
G\{\, 4,6 \,\} \big)$, and it is clear that each is a point so has
Euler characteristic $1$.

Finally, Equations \eqref{equ:kappa} and \eqref{equ:alpha_beta} and
diagram \eqref{equ:square} give 
\[
  \beta \big( [ S_{ 2,5 } ] \big)
  = \varepsilon \theta \big( [ S_{ 2,5 } ] \big)
  = \varepsilon \theta \kappa 
    \big( [ \overline{ S }_{ 2,5 } ] \big)
  = \varepsilon Q \overline{ \theta }
    \big( [ \overline{ S }_{ 2,5 } ] \big)
  = ( \ast \ast ).
\]
Equation \eqref{equ:exchange_triangles} gives exchange triangles
\[
  \xymatrix @R=1ex {
    \{\, 4,7 \,\} \ar[r] & \{\, 2,4 \,\} \oplus \{\, 5,7 \,\} \ar[r] & \{\, 2,5 \,\}
    & \{\, 2,5 \,\} \ar[r] & \{\, 2,7 \,\} \ar[r] & \{\, 4,7 \,\}
                   }
\]
and Equation \eqref{equ:overlinetheta} gives $\overline{ \theta }
\big( [ \overline{ S }_{ 2,5 } ] \big) = [ 2,4 ] + [ 5,7 ] -
[ 2,7 ]$ whence Equation \eqref{equ:epsilon} gives
\[
  ( \ast \ast )
  = \varepsilon Q
    \big( [ 2,4 ] + [ 5,7 ] - [ 2,7 ] \big)
  = vz.
\]

Hence Equation \eqref{equ:CC} says
\begin{align*}
  \rho \big( \{\, 4,6 \,\} \big)
  & = \alpha \big( \{\, 4,6 \,\} \big)
      \sum_e \chi \big( \Gr_e( G\{\, 4,6 \,\} ) \big) \beta( e ) \\
  & = z^{ -1 } \cdot \Big(
      \chi \big( \Gr_0( S_{ 2,5 } ) \big) 
         \beta( 0 )
      + \chi \big( \Gr_{[ S_{ 2,5 } ]}( S_{ 2,5 } ) \big) 
         \beta \big( [S_{ 2,5 }] \big)
                    \Big) \\
  & = z^{ -1 } \cdot ( 1 + vz ) \\
  & = \frac{ 1+vz }{ z }.
\end{align*}

Similar computations for the other indecomposable objects finally
produce the generalised frieze in Figure \ref{fig:generalised_frieze}
in the introduction.
\end{Example}

\section{Questions}
\label{sec:questions}

We end the paper with some questions.

\begin{enumerate}

\item  The group $\K_0^{ \split }( \sT )$ is free abelian on $\indec\,
  \sT$, and the subgroup $N$ of Definition \ref{def:N} is generated by
  all expressions $[a] - [a']$ where $s^{ \ast } \rightarrow a
  \rightarrow s$ and $s \rightarrow a' \rightarrow s^{ \ast }$ are
  exchange triangles with $s \in \indec\, \sT \setminus \indec\, \sR$.

\medskip
\noindent
What is the rank $n$ of the quotient $\K_0^{ \split }( \sT ) / N$?
Note that when $n$ is finite, it is the largest integer such that the
method of Remark \ref{rmk:app} results in a modified Caldero-Chapoton
map $\rho : \obj \sC \rightarrow \BZ[ x_1^{ \pm 1 }, \ldots, x_n^{ \pm
  1 } ]$ using all the variables $x_1, \ldots, x_n$.

\medskip

\item  Consider the $\BZ$-subalgebra of $\BZ[ x_1^{ \pm 1 }, \ldots, x_n^{
    \pm 1 } ]$ generated by the values of the modified Caldero-Chapoton map $\rho$. 

\medskip
\noindent
  What is its relation to the cluster algebra?

\medskip

\item  Let $T$ be a cluster tilting object and use it to define a
  Caldero-Chapoton map $X$.  If $T$ is subjected to cluster mutation,
  then the values of $X$ change in a well-understood way, see
  \cite[proof of cor.\ 5.4]{Palu}.

\medskip
\noindent
  There is a notion of mutation of rigid objects due to
  \cite[sec.\ 2]{MP}.  What happens to the values of the modified
  Caldero-Chapoton map under such mutation?

\end{enumerate}

\medskip
\noindent
{\bf Acknowledgement.}
This paper is a direct continuation of \cite{HJ}.   Both papers grew
out of \cite{BHJ} with Christine Bessenrodt, and we are grateful to
her for the fruitful collaboration.

We thank the referee for several interesting comments and Robert Marsh
for answering a question about rigid subcategories.

Part of this work was done while Peter J\o rgensen was visiting the
Leibniz Universit\"{a}t Hannover.  He thanks Christine Bessenrodt,
Thorsten Holm, and the Institut f\"{u}r Algebra, Zahlentheorie und
Diskrete Mathematik for their hospitality.  He gratefully acknowledges
support from Thorsten Holm's grant HO 1880/5-1, which falls under the
research priority programme SPP 1388 {\em Darstellungstheorie} of the
Deutsche Forschungsgemeinschaft (DFG).

\end{document}